\documentclass[reqno,11pt]{amsart}

\usepackage{amsmath,amssymb,amsthm,tikz,mathtools}
\usepackage{verbatim}
\usepackage{xcolor}
\usepackage{mathrsfs}
\usepackage{cancel}
\usepackage{hyperref}
\usepackage[normalem]{ulem}

\usepackage{centernot}
\usepackage{mathtools}

\usepackage{tikz-cd}
\usetikzlibrary{circuits.logic.US}
\usetikzlibrary{shapes.geometric}
\usetikzlibrary{graphs, arrows.meta, positioning}

\newtheorem{theorem}{Theorem}[section]
\newtheorem{definition}{Definition}[section]
\newtheorem{lemma}{Lemma}[section]

\newtheorem{proposition}{Proposition}[section]
\newtheorem{remark}{Remark}[section]

\numberwithin{equation}{section} 

\usepackage{setspace}
\numberwithin{equation}{section}

\title[Intrinsic scalings with non-standard growth]
{Intrinsic scalings with non-standard growth}
  
\author[M. Amaral]{Marcelo Amaral}
\address{UNILAB, ICEN, Universidade da Integração Internacional da Lusofonia Afro-Brasileira, 62.790-970,Redenção-CE, Brazil}{}
\email{marceloamaral@unilab.edu.br}

\author[J.  Araújo]{Janielly Araújo}
\address{UFERSA, DCETI, Universidade Federal Rural do Semi-Árido, 59515-000, Angicos-RN, Brazil}{}
\email{janielly@ufersa.edu.br}

\begin{document}




	
\begin{abstract}
In this work, we investigate quantitative regularity estimates for degenerate parabolic partial differential equations, with a focus on Orlicz-type diffusive structures. Using a geometric tangential analysis tailored to these structures and a general notion of intrinsic scalings, we derive precise interior H\"older regularity estimates for bounded weak solutions. These results offer new insights, even in the time-stationary case.
\end{abstract}
  
\date{\today}

\maketitle
 
\tableofcontents

\section{Introduction}

In this paper,  we study optimal regularity estimates for locally bounded solutions of the degenerate parabolic partial differential equations (PDEs), specifically in the context of Orlicz-Sobolev spaces whose prototype is
	\begin{equation}  \label{1}
		u_t- \textrm{div} \left(g(|\nabla u|) \frac{\nabla u}{|\nabla u|}\right) =f. 
	\end{equation}
Equations like \eqref{1} have been extensively studied across various contexts and hold significant importance in functional analysis and PDEs. The interest in these studies stems from the relevance of these problems to various fields, including biology, chemistry, and mathematical physics, where the nature of the model intrinsically dictates the geometry (growth) of the diffusion process. These studies extend classical spaces and are particularly useful in addressing problems where the growth conditions of the solutions are not necessarily polynomial.  This particular scenario corresponds to $p$-Laplacian equations -- case $g(t)\sim t^{p-1}$ -- which are given by
\begin{equation}  \nonumber
	\Delta_p u = \textrm{div} \left(|\nabla u|^{p-2}\nabla u\right), \quad \mbox{for } \; p>1.
\end{equation}

    The study of regularity estimates for weak solutions to parabolic equations, particularly those involving the \(p\)-Laplacian, has seen significant advancements over the past few decades, see references (\cite{AM07}, \cite{AMS04}, \cite{BC04}, \cite{DiBUV02}, \cite{DF84} \cite{daSS18} \cite{BDM13}, \cite{TU14}, \cite{K08}, \cite{KM11}, \cite{U08}, see also book \cite{U08}.  In the context of optimal H\"older regularity estimates, Teixeira and Urbano have introduced innovative techniques, emphasizing the significance of intrinsic scaling in regularity theory \cite{TU14} for weak solutions to degenerate equations
$$
u_t-\textrm{div} \left(|\nabla u|^{p-2}\nabla u\right)=f \in L^{q,r},
$$
where solutions are locally $C^{0,\alpha}$ in space for 
$$
\alpha=\frac{(p q-n) r-p q}{q[(p-1) r-(p-2)]}.
$$
Under stronger integrability assumptions on the source term, sharp gradient regularity estimates were further established in \cite{Amaral}. The approach builds upon the techniques developed in \cite{C^p', JMRN}, where refined gradient oscillation estimates were obtained for the elliptic case, relying on a positive answer of the \(C^{p'}\)-regularity conjecture in certain specific cases. 

\subsection{The main result}
The primary goal of this work is to establish optimal H\"older regularity estimates for weak solutions of a broader class of degenerate parabolic PDEs, as given in \eqref{1}, where the diffusion term \( g \) exhibits Orlicz-type growth  
\begin{equation} \label{elba}
1 < g_0 \le \frac{s \cdot g'(s)}{g(s)} \le g_1, \quad s > 0.  
\end{equation}  

In parallel, we consider the inhomogeneity term \( f \in L^{F,r}(\Omega_T) \), where  
\[
L^{F,r}(\Omega_T) := L^{r}(0, T; L^{F}(\Omega)),  
\]  
equipped with the norm  
\[
\|f\|_{L^{F,r}(\Omega_{T})} := \left( \int_{0}^{T} \|f(\cdot, t)\|_{L^{F}(\Omega)}^r \, dt \right)^{\frac{1}{r}}.  
\]  
Here, \( \|\cdot\|_{L^F(\Omega)} \) denotes the Luxemburg norm (see \eqref{existence condition} for a detailed setup). In accordance with the Orlicz-type growth framework, we assume that
\begin{equation} \label{ivete} 
0 <f_0 \le \frac{s \cdot F''(s)}{F'(s)} \le f_1, \quad s > 0.  
\end{equation}

Finally, we observe that our present framework reduces to the case studied in \cite{TU14} when considering polynomial-type functions \( g(t) = t^{p-1} \) and \( F(t) = t^q \). 

With the preliminary setup in place, we now present our main result, where \( G \in \mathcal{G}_{g_{0},g_{1}} \) and \( F \in \mathcal{G}_{f_{0},f_{1}} \) are related to the assumptions \eqref{elba} and \eqref{ivete}, respectively. For the complete problem setup, refer to Section \ref{caucaia}.

\begin{theorem}\label{principal}
Let \( u \) be a locally bounded weak solution of \eqref{1}. Setting \( g = G' \) and \( f \in L^{F,r}(Q_1) \), for \( G \in \mathcal{G}_{g_{0},g_{1}}\) and \( F \in \mathcal{G}_{f_{0},f_{1}} \). Then \( u \) is locally \( \alpha \)-H\"older continuous in space with exponent  
\begin{equation}  \label{optimal alpha}  
\alpha = \dfrac{[(g_{0}+1)(f_{0}+1)-n]r - (g_{0}+1)(f_{0}+1)}{(f_{0}+1)\left[g_{0}r-(g_{0}-1)\right]}.  
\end{equation}  
Moreover, \( u \) is locally \( \beta \)-H\"older continuous in time with \( \beta = \alpha / \theta \), where  
\begin{equation} \label{teta}  
\theta = 1+\alpha - (\alpha-1)g_1  
\end{equation}  
for some universal constant \( \rho > 0 \) sufficiently small.  Furthermore, there exists a constant \( C > 0 \), depending only on \( g(1), g_0, g_1, r, f_0 \), and \( f_1 \), such that  
\begin{equation}\label{regestimate}  
\sup_{B_\rho \times (-\rho^\theta,0] } |u(x,t) - u(0,0)| \leq C\rho^\alpha.  
\end{equation}  
\end{theorem}

\subsection{Moving parabolic $g$-cylinders}
At a closer look, equation \eqref{1} reveals a fundamental challenge related to the highly irregular behavior of intrinsic cylinders, which plays a crucial role in obtaining optimal \(\alpha\)-H\"older estimates, see \cite{DiB93,U08}. In the case of the \( p \)-Laplacian, intrinsic cylinders take the form  
\[
\mathcal{Q}_\rho = B_\rho \times (-\rho^\theta, 0], \quad \text{where} \quad \theta = 1 + \alpha - (p-1)(\alpha - 1).  
\]  
However, for parabolic equations with Orlicz structures, a corresponding generalized intrinsic scaling must be defined according to the structure of the diffusion growth function \( g \), which is given by  
\begin{equation}\label{thetaJPA}
\theta(\rho) = \theta_{g,\alpha}(\rho) = 1 + \alpha - \log_{\rho}(g(\rho^{\alpha - 1})).
\end{equation}
Under this perspective, for each radii $ \rho >0$ small enough, we shall consider cylinders
\begin{equation}\label{cylJPA}
Q_\rho = B_{\rho} \times (-\rho^{\theta(\rho)}, 0\,].
\end{equation}
Such choices involve adapting the scaling to the specific characteristics of equation \eqref{1}, effectively balancing the temporal and spatial components in accordance with the generalized structure determined by the prescribed diffusion rate \( g \). For a more detailed analysis, refer to Remark \ref{g(1)}.

\subsection{Some emblematic cases}
Interesting examples of Orlicz functions covered in this article include \( g(t) = t^{\beta} \ln (\gamma t + \eta) \), for positive parameters \(\beta, \gamma\) and \(\eta \). In this case, \( g_{0} = \beta \) and \( g_{1} = \beta + 1 \). If $g(t) = t^{p-1}$ with $g_{0} = g_{1}= p-1$, we obtain the prototype of p-Laplacian. There are also interesting and different examples $g(t) = t^{\beta} \ln (\gamma t + \eta)$, with $\beta, \gamma , \eta >0$ and $g_{0} = \beta$ and $g_{1} = \beta+1$ or by discontinuous power transitions like
	$$
	g(t) = \left\{
	\begin{array}{rcl}
	c_1 t^{\beta},& \mbox{if} & 0 \le t \leq t_0\\
	c_2 t^{\gamma} +c_3, & \mbox{if} & t \ge t_0
	\end{array}
	\right.
	$$
where $\beta, \gamma, t_0$ are positive numbers, and $c_1,c_2, c_3$ are real numbers such that $g \in C^1([0, \infty))$ with $g_0 = \min(\beta, \gamma)$ and $g_1 = \max(\beta, \gamma)$, among others. This class of nonlinear evolution equations appear in many relevant applications of physics, fluid dynamics and image processing, for instance \cite{Di},\cite{Fisica}.

We emphasize that our result is in accordance with well-known estimates obtained. For instance, our result generalizes the optimal regularity exponent $\alpha$ provided in \cite{TU14} in which was addressed the inhomogeneous p-laplace parabolic equation
\begin{equation} \nonumber
	u_t - \textrm{div} \left(g(|\nabla u|) \frac{\nabla u}{|\nabla u|}\right) = f \in L^{F,r}, \,\,\ F\in \mathcal{G}_{q-1,q-1}
\end{equation}
where $g(t)=t^{p-1}$ with $g_{0}=g_{1}=p-1$
and obtained
$$
\alpha   = \dfrac{(pq-n)r - pq}{q\left[(p-1)r-(p-2)\right]}.
$$
Also, as a consequence of Theorem \ref{principal} the optimal regularity exponent of inhomogeneous parabolic equation
\begin{equation} \nonumber
	u_t - \textrm{div} \left(g(|\nabla u|) \frac{\nabla u}{|\nabla u|}\right) = f \in L^{F,r}, \,\,\ F\in \mathcal{G}_{f_{0},f_{1}}
\end{equation}
where $g(t) = t^{\beta} \ln (\gamma t + \eta)$, $\beta, \gamma, \eta > 0$ with $g_{0}=\beta$ and $g_{1}= \beta -1$
is 
$$
\alpha   = \dfrac{[(\beta+1)f_{0}-n]r - (\beta+1)f_{0}}{f_{0}\left[\beta r-(\beta-1)\right]} .
$$

The paper is organized as follows. In section \ref{caucaia} we consider techniques and assumptions related to our framework. In section \ref{preliminares}, we give some definitions and develop some preliminary estimates that  will be used throughout the paper. In section \ref{compacidade}, we establish compactness methods and intrinsic scalings. In section \ref{principal}, we prove the main result of the paper. 

\section{Mathematical setup}\label{caucaia}

Before presenting our main result, we introduce technical notations and assumptions about our framework. First, we shall consider \( Q_1 := B_1 \times (-1,0]  \) is the unit parabolic cylinder, where \( B_r = B_r(0) \) is the \( n \)-dimensional ball centered at the origin with radius $r>0$. 

\subsection{Assumptions on $g$}
From \eqref{1}, we shall assume, according to Lieberman \cite{Lieberman}, that \( g \in C^{0}([0, +\infty)) \cap C^{1}((0, +\infty)) \) and satisfies the following growth condition: for given parameters \( g_0, g_1 \in \mathbb{R}_+ \),
\begin{equation} \label{LU}
1 < g_0 \le \frac{s \cdot g'(s)}{g(s)} \le g_1, \quad s > 0.
\end{equation}

As noted in \cite{Lieberman}, for a given N-funtion $G$, we define
$$
L^{G}(\Omega_{T}) := \left\{u:\Omega_{T} \rightarrow \mathbb{R} \,\,\, \textrm{measurable}: \,  \int_{\Omega_{T}} G(|u(x,t)|) dxdt  < \infty\right\}.
$$
The Orlicz-Sobolev space is expressed as
$$
W^{1,G}(\Omega_{T}) := \{u \in L^{G}(\Omega_{T}) : \nabla u \in L^{G}(\Omega_{T})\},
$$
which are endowed with the respective Luxemburg norms
$$
\|u\|_{L^{G}(\Omega_{T})} = \inf \left\{  \lambda > 0 : \,\,\ \int_{\Omega_{T}} G\left(\frac{|u(x,t)|}{\lambda}\right) dxdt \le 1   \right \}.
$$
and
$$
\|u\|_{W^{1,G}(\Omega_{T})} := \|u\|_{L^{G}(\Omega_{T})} + \|\nabla u\|_{L^{G}(\Omega_{T})}.
$$
Given \( G : [0, +\infty) \to \mathbb{R} \) the primitive function of \( g \), i.e.,
\begin{equation} \label{pri1}
G'(s) = g(s),
\end{equation}
we consider the following class of functions
$$
\mathcal{G}_{g_{0},g_{1}}:= \{G \; | \; G'=g \in C^0([0,+\infty)) \cap C^1((0,+\infty))  \; \mbox{and}\;  g \,\,\, \textrm{satisfies \eqref{LU}} \}.
$$

\begin{remark}
Observe the if \( G \in \mathcal{G}_{g_0, g_1} \), then \( \lambda G \in \mathcal{G}_{g_0, g_1} \) for any \(\lambda > 0\). This property turns out to be relevant because multiplier terms are related to ellipticity conditions. This can be fully observed in the classical linear case \( g(s) = \lambda s \), where the identity 
$$
\text{div} \left(g(|\nabla u|) \frac{\nabla u}{|\nabla u|}\right) = \lambda \Delta u
$$
shows that \(\lambda\) acts as the ellipticity constant, indicating that regularity estimates for solutions should depend on \(\lambda = g(1)\). To raise such dependence for generalized structures on $g$, we include in our analysis the following growth parameter: select \( s_0 \in g^{-1}(\mathbb{R}^+) \) and set \( g(s_0) =: \lambda > 0 \). See Remark \ref{remJPA} for further considerations.
\end{remark}

In light of this, for the sake of simplicity, we will assume that
\begin{equation}
g(1) = 1.  
\end{equation}

\subsection{Assumptions on $f$} Concerning the structural hypothesis on the inhomogeneity term \( f \), we make the following considerations. For a parameter \( T > 0 \) and a given domain \( \Omega \subset \mathbb{R}^n \), we denote \(\Omega_T = \Omega \times (-T, 0]\). For a given $F\in \mathcal{G}_{f_0,f_1}$, where $0<f_0\leq f_1$, and \( r > 1 \), we shall assume the inhomogeneity term \( f \in L^{F,r}(\Omega_T) \), where \( L^{F,r}(\Omega_T) := L^{r}(0, T; L^{F}(\Omega)) \), endowed with the following norm
\[
\|f\|_{L^{F,r}(\Omega_{T})} := \left( \int_{0}^{T} \|f(\cdot, t)\|_{L^{F}(\Omega)}^r \, dt \right)^{\frac{1}{r}},
\]
where we denote $\|\cdot\|_{L^F(\Omega)}$ the Luxemburg norm for spatial functions
$$
	\|f\|_{L^F(\Omega)} = \inf \left\{\kappa >0 : \,\, \int_{\Omega} F \left(\frac{|f(x)|}{\kappa}\right) dx \le 1\right\}.
$$

Throughout the paper, we assume the following compatibility condition:
\begin{equation} \label{existence condition}
\frac{1}{r}+\dfrac{n}{(f_0+1)(g_0+1)} < 1 < \frac{2}{r} + \dfrac{n}{(f_0+1)}.
\end{equation}
The inequality on the left establishes the minimal integrability condition necessary to ensure the existence of bounded weak solutions to \eqref{1}. Meanwhile, the inequality on the right serves as the lower bound that guarantees weak solutions are H\"older continuous in space and time.

\section{Definitions and results in Orlicz spaces} \label{preliminares}

In this section we present some definitions and results which we will use in this paper. We begin by recalling the definition of N-function.

\definition{We say that $G$ is an N-function if $G(t)= \int_{0}^{t} g(s)ds$ and $g: [0, \infty) \rightarrow \mathbb{R}$ is a positive nondecreasing function such that $g(0)=0$, $\lim_{t \rightarrow \infty} g(t) = \infty$ and $\lim_{s \rightarrow t^{+}} g(s) = g(t)$}.

The complementary function of an N-function $G$ is given by
$$
\tilde{G} (t)= \int_{0}^{t} \tilde{g}(s)ds
$$ 
where $\tilde{g}(s) := \sup\{t; g(t)\le s\}$. As an immediate consequence of \eqref{LU} we have the following properties for $g$, $G$, $\tilde{g}$ and $\tilde{G}$.  

\begin{lemma} \label{p1}
   Let $G$ be an N-function satisfying \eqref{LU} and \eqref{pri1} and $\tilde{G}$ its complementary function, so we have  for all $s,t > 0$:
\begin{itemize} 

\item[(a)] $\min\{s^{g_0}, s^{g_1}\} g(t) \le g(st) \le \max\{s^{g_0},s^{g_1}\} g(t)$;

\item[(b)] $ \min\{s^{1+g_0}, s^{1+g_1}\} G(t) \le G(st) \le  \max\{s^{1+g_0},s^{1+g_1}\} G(t)$;

\item[(c)] $ \frac{tg(t)}{1+g_{1}} \le G(t) \le tg(t)$,

\item[(d)] $\min\{s^{\frac{1}{g_0}}, s^{\frac{1}{g_1}}\} \tilde{g}(t) \le \tilde{g}(st) \le \max\{s^{\frac{1}{g_0}},s^{\frac{1}{g_1}}\} \tilde{g}(t)$,

\item[(e)]$\min \{s^{1+\frac{1}{g_0}},s^{1+\frac{1}{g_1}}\} \widetilde{G}(t) \le \tilde{G}(st) \le  \max\{s^{1+\frac{1}{g_0}}, s^{1+\frac{1}{g_1}}\} \widetilde{G}(t),$

\item[(f)] $\widetilde{G}(g(t)) \le g_1 G(t)$.

\end{itemize}

\end{lemma}

\begin{remark}\label{g(1)} As we added the normalization condition $g(1)=1$ in the lemma above, it follows that
\begin{itemize}
    \item [(g)] $\min\{s^{g_0}, s^{g_1}\} \le g(s) \le \max\{s^{g_0},s^{g_1}\}$.
\end{itemize}
Under scenario \( 0 < \alpha < 1 \), we obtain (see Remark \eqref{g(1)} below)
$$
\rho^{(\alpha - 1) g_0} \le g(\rho^{\alpha - 1}) \le \rho^{(\alpha - 1) g_1},
$$
which leads to the inequality
$$
(\alpha - 1) g_1 \le \log_\rho (g(\rho^{\alpha - 1})) \le (\alpha - 1) g_0.
$$
Therefore,
$$
1+\alpha - (\alpha-1)g_0 \le \theta(\rho) \le 1+\alpha - (\alpha-1)g_1.
$$
This implies that \(\theta(\rho)>2\), provided $g_0 >1$. On the other hand, from \eqref{thetaJPA}, we deduce 
$$
\rho^{\theta(\rho)} = \rho^{1+\alpha} g(\rho^{\alpha-1})^{-1}.
$$
Then, cylinder \eqref{cylJPA} can be rewritten as follows
\begin{equation}\label{cylJPA2}
Q_\rho = B_{\rho} \times (-\rho^{1+\alpha}g(\rho^{\alpha-1})^{-1}, 0\,].
\end{equation}
From this, we characterize the geometry of the intrinsic parabolic cylinders under the generalized diffusion growth perspective. Consequently, we notice that cylinders \(Q_\rho\) are radially increasing.
\end{remark}

\begin{lemma} \label{constante}
There is a constant $C=C(g_{0},g_{1})$ such that,
$$
\|u\|_{L^{G}(\Omega_{T})} \le \max \left\{ \left(\int_{\Omega_{T}}G(|u|)dx dt \right)^{\frac{1}{1+g_{0}}} , \left(\int_{\Omega_{T}} G(|u|)dx dt \right)^{\frac{1}{1+g_{1}}} \right\}.
$$
\end{lemma}
\begin{proof}
See Lemma 2.3 in \cite{MW}.
\end{proof}

Now, we give an appropriate notion of bounded weak solutions for equation \eqref{1}.
\definition{We will say a function $u(x,t)\in L^{G}_{loc}(0,T;W^{1,G}_{loc}(\Omega))$ if $u(x,t)\in L^{p}_{loc}(0,T;W^{1,p}_{loc}(\Omega))$ and $\iint_{\Omega_{T}} G(\nabla u) dxdt < \infty$. A locally bounded function $u(x,t) \in C_{loc} (0, T;L_{ loc}^{2}(\Omega)) \cap L^{G}_{loc}(0,T;W^{1,G}_{loc}(\Omega))$
is a weak solution to $(\ref{1})$ in a space-time cylinder $K \times [t_{1}, t_{2}] \subset \Omega_{T}$ if we have
\begin{eqnarray} \nonumber
\int_{K}u\varphi dx \left| \right._{t_{1}}^{t_{2}} + \int_{t_{1}}^{t_{2}}\int_{K}\left\{-u\varphi_{t} + g(|\nabla u|) \frac{\nabla u}{|\nabla u|} . \nabla \varphi \right\}dxdt =\int_{t_{1}}^{t_{2}}\int_{K}f\varphi dxdt
\end{eqnarray}
for every $\varphi \in L^{G}_{loc}(0,T;W^{1,G}_{0}(K))$ and $\varphi_{t} \in L^{2}_{loc}(\Omega_{T})$.

An alternative definition of weak solutions using the Steklov average of a function $u \in L^{1}(\Omega_{T})$ given by
$$
u_{h} =
 \begin{cases}

    \frac{1}{h}\int_{t}^{t+h} u( \cdot,\tau) d\tau , & \mbox{if } t\in (0,T-h] \\

    0, & \mbox{if } t\in (T-h,T].
\end{cases}
$$
allows us to prove the following Caccioppoli-type energy estimate which is a key ingredient in our analysis. It will be responsible for obtaining compactness in the iteraction process later.

\begin{proposition} \label{energia}
Let $u$ be a weak solution for \eqref{1} in $\Omega_{T}$ and $K \times [t_{1}, t_{2}] \subset \Omega \times (0,T]$. There exists a constant $C$, depending only on $n, g_{0}, g_{1}$, $K \times [t_{1}, t_{2}]$  such that
{\small
$$
\sup_{t_{1} \le t \leq t_{2}}\int_{K} u^{2} \psi^{1 + g_{1}} dx + \int_{t_{1}}^{t_{2}} \int_{K} g(|\nabla u|) \nabla u \psi^{1 + g_{1}} dx d\tau \leq  C\int_{t_{1}}^{t_{2}}\int_{K} u^{2}\psi^{g_{1}} \psi_{t} dx dt +
$$
}
$$
+ \int_{t_{1}}^{t_{2}}\int_{K} \max\{(|\nabla \psi|)^{1+g_{0}}, (|\nabla\psi|)^{1+g_{1}} \} G(|u|)  dxdt + C\Vert f \Vert^{2}_{L^{G,r}}
$$
for all $\psi \in C_{0}^{\infty}(K \times (t_{1}, t_{2}))$ such that $\xi \in [0, 1].$ 
\end{proposition}
\begin{proof}
 Choose $\xi= u_h \psi^{1+g_{1}}$ as a test function and perform the usual combination of integrating in
time, passing to the limit in $h \rightarrow 0$, applying Young and H\"older inequality, Lemma \ref{p1} and Lemma 2.7 in \cite{Fernandes} to derive the estimate.
\end{proof}

\section{The interactive oscilation improvment} \label{compacidade}

We begin with an approximation result, which serves as the principal argument in the key lemma below. Roughly speaking, it states that if a normalized solution of the inhomogeneous problem has a sufficiently small source term, then the solution can be made arbitrarily close to a solution of the homogeneous problem. The proof follows as a consequence of Lemma \ref{energia}.  

To simplify our analysis without loss of generality, we assume the normalization condition \( g(1) = 1 \) in the proof of Theorem \ref{1}. This assumption can always be enforced by defining  
\begin{equation}\label{remJPA}  
\tilde{g}(s) := \frac{g(s_0 s)}{\lambda}.  
\end{equation}  
Notably, we have \( \tilde{g} \in \mathcal{G}_{g_0, g_1} \), ensuring that the regularity estimates remain dependent on \( \lambda \).  
	
\begin{lemma} \label{comp}
		Given $\delta >0$, there exists $0 < \epsilon \ll 1$ such that if $\|f\|_{L^{F,r}(Q_{1})} \le \epsilon$ where $G \in \mathcal{G}_{f_0,f_1}$
and a weak solution for \eqref{1} in $Q_{1}$, where $F\in \mathcal{G}_{f_0,f_1}$ with $||u||_{\infty, Q_1} \le 1$, then we can find a function $h$ solution of
 \begin{equation} \label{homogeneous equation}
  h_{t} - \textrm{div} \left(g(|\nabla h|) \frac{\nabla h}{|\nabla h|}\right) =0 \quad \textrm{in} \quad Q_{1/2}
  \end{equation}
so that
$$
\|u-h\|_{\infty,Q_{1/2}} \le \delta.
$$
\end{lemma}

\begin{proof}
		Suppose by purpose of contradiction that the thesis of the Lemma fails. Then, there would exist $\delta_0>0$ so that we could find sequences of functions 
        $$
        (u^{k})_{k} \in C_{loc} (0, T;L_{ loc}^{2}(B_{1})) \cap L^{G}_{loc}(0,T;W^{1,G}_{loc}(B_{1}))
        $$
        where $G \in \mathcal{G}_{g_0,g_1}$ and $(f^{k})_{k} \in L^{F,r}(Q_1)$ where $F \in \mathcal{G}_{f_0,f_1}$ satisfying 
  \begin{equation}
 \label{bounded}
  \displaystyle ||u^k||_{ \infty, Q_{1}} \le 1
 \end{equation}
  and linked through 
		\begin{equation} \label{seq}
			(u^{k})_{t} - \textrm{div} \left(g(|\nabla u^k|) \frac{\nabla u^k}{|\nabla u^k|}\right) =f^{k} \quad \textrm{in} \quad Q_1,
		\end{equation}
		in the weak sense and
		$$
     ||f^{k}||_{L^{F,r}(Q_{1})} \leq \frac{1}{k}.
		$$
		However,
		\begin{equation}\label{4}
			||u^k-h||_{\infty,Q_{1/2}} > \delta_0
		\end{equation}
		for all $h$ weak solution of \eqref{homogeneous equation}. From the Caccioppoli estimate, Lemma \ref{energia} to get
\begin{equation}
\begin{array}{c}
\displaystyle\int_{-1}^{0}\int_{B_{1}} G(|\nabla u_{k}|)\psi^{1+g_1}dxdt \\ 
\le \displaystyle \sup_{t_{1} \leq t \leq t_{2}}\int_{K} u_{k}^{2} \psi^{1 + g_{1}} dx + \int_{t_{1}}^{t_{2}} \int_{K} g(|\nabla u_{k}|) \nabla u_{k} \psi^{1 + g_{1}} dx dt \\
\le \displaystyle C\int_{t_{1}}^{t_{2}}\int_{K} u_{k}^{2}\psi^{g_{1}} \psi_{t} dxdt \\
+ \displaystyle\int_{t_{1}}^{t_{2}}\int_{K} \max\{(|\nabla \psi|)^{1+g_{0}}, (|\nabla\psi|)^{1+g_{1}} \} G(|u_{k}|)  dxdt + C\Vert f_{k} \Vert^{2}_{L^{F,r}} \\ 
\le \tilde{C}.
\end{array}
\end{equation}
Thus, we obtain,
\begin{eqnarray*}
\int_{-1/2}^{0}\int_{B_{1/2}} G(|\nabla u_{k}|)dxdt &\le& \int_{-1}^{0}\int_{B_{1}} G(|\nabla u_{k}|)\psi^{1+g_1}dxdt \le \tilde{C}
\end{eqnarray*}
and by the Lemma \ref{constante}, we have
$$
\|\nabla u_{k} \|_{L^{G}(Q_{1/2})} \le \tilde{C}.
$$
Since $L^{G}$ is a reflexive space, up to a subsequence, 
 $$
 \nabla u^{k} \rightharpoonup \xi
 $$ 
 in $L^{G}(Q_{1/2})$. Moreover, by \eqref{bounded} $(u^{k})_{k}$ is a equibounded  and also equicontinuous sequence, see \cite{HR}. By Arzela-Ascoli Theorem, along a subsequence
 $$
 u^{k} \rightarrow \phi,
 $$
 uniformly in $Q_{1/2}$.
 Also, we can identify $\xi = \nabla \phi$, passing to the limit in \eqref{seq}  which contradicts \eqref{4} if we choice $h=\phi$ for $k \gg 1$.
\end{proof}
 By the approximation Lemma above  we can prove the first iterative step in order to obtain the H\"older regularity  desired.
	\begin{lemma} \label{passo1}
		Let $u$ be a weak solution for \eqref{1} in $Q_{1}$
		with $||u||_{\infty,Q_{1}} \le 1$. There exists $\epsilon >0$ and $0 < \rho \ll 1/2,$ such that if
		$ \|f \|_{L^{F,r}(Q_{1})} \le \epsilon$ where $F \in \mathcal{G}_{f_0,f_1}$, then
		$$
		\underset{Q_{\rho}}{\sup} \vert u(x,t) - u(0, 0) \vert \leq \rho^{\alpha}.
		$$
		\end{lemma}
	
	\begin{proof}
		By Lemma \ref{comp}, for a $\delta>0$ to be chosen later, there is a weak solution for \eqref{homogeneous equation} such that
		 $$
		\displaystyle	||u-h||_{\infty,Q_{1/2}} \le \delta.
	   $$
		From the available regularity theory (see \cite{Grad}) $h$ is locally $C_{x}^{0,1} \cap C_{t}^{0,\frac{1}{2}}$. Thus for $(x,t) \in Q_{\rho}$ 
\begin{eqnarray*}
\vert h(x, t) - h(0, 0) \vert &\leq &  \vert h(x, t) - h(0, t) \vert + \vert h(0, t) - h(0, 0) \vert \\
&\leq & c_{1}\vert x - 0 \vert +  c_{2}\vert t - 0 \vert^{\frac{1}{2}}\\
&\leq & c_{1}\rho + c_{2}\rho^{\frac{\theta}{2}}\\
&\leq & C\rho
\end{eqnarray*}
since $\theta > 2$ and $\rho < 1$. Therefore, we have
\begin{eqnarray*}
\underset{(x, t) \in Q_{\rho}}{\sup} \vert h(x, t) - h(0, 0) \vert \leq C\rho,
\end{eqnarray*}
where the constant $C > 0$ is universal. We will choose $\rho \ll 1/2$ so that holds the estimate
		\begin{eqnarray*}
		  \underset{Q_{\rho}}{\sup} \vert u(x,t) - u(0, 0) \vert &\leq& \underset{Q_{1/2}}{\sup} \vert u(x,t) - h(x,t) \vert  + |u(0,0)-h(0,0)| + \\
          &+&\underset{Q_{\rho}}{\sup} \vert h(x,t) - h(0, 0) \vert \\
    &\leq & 2 \delta + C\rho.
		\end{eqnarray*}
		Finally, we choose $\rho \ll \frac{1}{2}$ so smal that
		
		\begin{equation*} 
			C \rho \leq \frac{1}{2} \rho^{\alpha}
		\end{equation*}
		and we define
		\begin{equation*} 
			\delta := \frac{1}{4} \rho^{\alpha}. 
		\end{equation*}
	  As a consequence, we conclude that
		$$
		\underset{Q_{\rho}}{\sup} \vert u(x,t) - u(0, 0) \vert \leq \rho^{\alpha}.
		$$

\end{proof}
In the sequel, we shall iterate Lemma \ref{passo1} in the appropriate geometric setting.
 
	\begin{lemma} \label{iterative}
		Under the hypotheses of Lemma \ref{passo1}, we have
		\begin{equation} \label{p_k}
        \underset{Q_{\rho^{k}}}{\sup} \vert u(x,t) - u(0, 0) \vert \leq {(\rho^{k})}^{\alpha}.
		\end{equation}
	\end{lemma} 
	
	\begin{proof}
		The proof is given by induction process. The first step, $k=1$ holds due to the previous Lemma. Now, suppose that \eqref{p_k} is true for $k$ and define the function $v_{k}: Q_{1} \rightarrow \mathbb{R}$ by
		\begin{equation}\label{scaling}
		 \quad v_{k}(x,t) = \frac{u(\rho^{k} x, \rho^{k \theta}t) - u(0,0)}{\rho^{k \alpha}}.
		\end{equation}
		Notice that for a positive constant $d$, the function $g_{d}(t) : = \frac{g(dt)}{g(d)}$ satisfies \eqref{LU}
		with the same constants $g_0$ and $g_1$.
		Then, for $d: = \rho^{k(\alpha -1)}$ we get
		$$
(v_{k})_{t}(x,t) - \textrm{div} \left(g_{d}(|\nabla v_{k}(x,t)|) \frac{\nabla v_{k}(x,t)}{|\nabla v_{k}(x,t)|}\right) =
$$
$$
  = \rho^{k\theta-k\alpha}u_{t}(\rho^{k} x, \rho^{\theta k}t) - \frac{\rho^{k}}{g(\rho^{k(\alpha -1)})}\textrm{div} \left(g(|\nabla u (\rho^k x,\rho^{k\theta}t)|) \frac{ \nabla u (\rho^k x,\rho^{k\theta}t)}{|\nabla  u (\rho^kx,\rho^{k\theta}t)|}\right).
$$
Choosing 
$$
\theta := (1+\alpha) - \log_{\rho^{k}}(g(\rho^{k(\alpha-1)}))
$$
we obtain
\begin{eqnarray*}
(v_{k})_{t}(x,t) - \textrm{div} \left(g_{d}(|\nabla v_{k}(x,t)|) \frac{\nabla v_{k}(x,t)}{|\nabla v_{k}(x,t)|}\right) &=& \frac{\rho^{k}}{g(\rho^{k(\alpha -1)})} f(\rho^k x,\rho^{k\theta}t) \\
&:=&  \tilde{f}(x,t).
\end{eqnarray*}
Let us observe that
{\small \begin{eqnarray*}
    \|f( \rho^{k}x,\rho^{k\theta}t)\|_{L^{F}(B_{1})} &=& \inf \left\{  \lambda > 0 : \,\,\ \int_{B_{1}} F\left(\frac{|f(\rho^{k}x,\rho^{k\theta}t)|}{\lambda}\right) dx \le 1   \right \} \\
    &=& \inf \left\{  \lambda > 0 : \,\,\ \int_{B_{\rho^{k}}} F\left(\frac{|f(x,\rho^{k\theta}t)|}{\lambda}\right) \rho^{-nk}dx \le 1   \right \} \\
    &\le& \inf \left\{  \lambda > 0 : \,\,\ \int_{B_{\rho^{k}}} F\left(\frac{\rho^{\frac{-nk}{1+f_0}}|f(x,\rho^{k\theta}t)|}{\lambda}\right) dx \le 1   \right \} \\
    &=& \inf \left\{  \rho^{\frac{-nk}{1+f_0}} \lambda > 0 : \,\,\ \int_{B_{\rho^{k}}} F\left(\frac{|f(x,\rho^{k\theta}t)|}{\lambda}\right) dx \le 1   \right \} \\
    &=& \rho^{\frac{-nk}{1+f_0}} \|f( x,\rho^{k\theta}t)\|_{L^{F}(B_{\rho^{k}})}.
\end{eqnarray*}}
So we have

\begin{equation*}
\begin{array}{c}
\int_{B_{1}} F\left(  \frac{|\tilde{f}(x,t)|}{\frac{\rho^{k}}{g(\rho^{k(\alpha-1)})}\rho^{\frac{-kn}{1+f_0}}\|f(x,\rho^{k\theta}t)\|_{L^{F}(B_{\rho^{k}})}}   \right) dx \\ 
=\int_{B_{1}} F \left(   \frac{|f(\rho^k x,\rho^{k\theta}t)|}{\rho^{\frac{-kn}{1+ f_0}}\|f(x,\rho^{k\theta}t)\|_{L^{F}(B_{\rho^{k}})}}\right) dx \\
\le  \int_{B_{1}} F \left(   \frac{|f(\rho^{k}x,\rho^{k\theta}t)|}{\|f( \rho^{k}x,\rho^{k\theta}t)\|_{L^{F}(B_{1})}}\right) dx \\
\le 1.
\end{array}
\end{equation*}

Therefore
\begin{eqnarray*}
||\tilde{f}||_{L^{F,r}(Q_{1})} &=& \left(\int_{-1}^{0}\| \tilde{f}(x,t) \|_{L^{F}(B_{1})}dt \right)^{\frac{1}{r}} \\
&\le& \frac{\rho^{k}}{{g(\rho^{k(\alpha-1)})}} \rho^{\frac{-kn}{1+f_0}}{\left(\int_{-1}^{0}\|f(x,\rho^{k\theta}t)\|_{L^{F}(B_{\rho^{k}})}dt\right)}^{\frac{1}{r}} \\
&=& \frac{\rho^{k[1-(\frac{n}{1+f_0}+\frac{\theta}{r})]}}{{g(\rho^{k(\alpha-1)})}} ||f||_{L^{F,r}(Q_{\rho^{k}})} \\
&=& \frac{\rho^{k[1-(\frac{n}{1+f_0}+\frac{1+\alpha}{r})]}(g(\rho^{k(\alpha-1)}))^{\frac{1}{r}}}{{g(\rho^{k(\alpha-1)})}} ||f||_{L^{F,r}(Q_{\rho^{k}})}  \\
&\le& \rho^{k[1-(\frac{n}{1+f_0}+\frac{1+\alpha}{r})+(\alpha-1)g_{0}\left(\frac{1}{r}-1\right)]} ||f||_{L^{F,r}(Q_{1})}.
\end{eqnarray*}
To apply the previous Lemma we need to have
$$
 1-\left(\frac{n}{1+f_0}+\frac{1+\alpha}{r}\right)+(\alpha-1)g_{0}\left(\frac{1}{r}-1\right) \ge 0 
$$
Therefore, we put
\begin{equation}
 \alpha   = \dfrac{[(g_{0}+1)(f_0+1)-n]r - (g_{0}+1)(f_0+1)}{(f_0+1)\left[g_{0}r-(g_{0}-1)\right]}   
\end{equation}
and so
 $$
 \| \tilde{f}(x,t) \|_{L^{F,r}(G_{1})} \le \varepsilon.
 $$
 Also, notice that
			\begin{eqnarray*}
			 ||v_k||_{\infty, Q_1} &=&  \left| \left| \frac{|u(\rho^k x, \rho^{k\theta}t) - u(0,0)|}{\rho^{k \alpha}}\right| \right|_{\infty, Q_1} \\
			&=&   \rho^{-k \alpha}\displaystyle \left||u(x,t)-u(0,0)|\right|_{\infty, Q_{\rho^{k}}} \\
			&\le&
			\rho^{-k \alpha} \rho^{k \alpha }\\
			& \le & 1.
		\end{eqnarray*}
	Therefore, we can apply Lemma \ref{passo1} to the function $v_k$ and being $v_k(0,0)=0$ we obtain
		\begin{eqnarray*}
		\rho^{\alpha}& \ge & \underset{Q_{\rho}}{\sup} \vert v_{k}(x,t) - v_{k}(0, 0) \vert \\
		& = & \underset{Q_{\rho}}{\sup}\frac{|u(\rho^k x,\rho^{k \theta}t) - u(0,0)|}{\rho^{k\alpha}} \\
		& = & \underset{Q_{\rho^{k+1}}}{\sup}\frac{|u(x,t) - u(0,0)|}{\rho^{k\alpha}} .
		\end{eqnarray*}
	Thus, we obtain
	$$
	\underset{Q_{\rho^{k+1}}}{\sup} |u(x,t) - u(0,0)| \le \rho^{(k+1) \alpha},
	$$
and the proof of the Lemma is finished.
	\end{proof}

\section{Proof of the main result} \label{principal}
 
In this section we will prove the our main result, providing with no loss of generality, optimal Holder continuity at the origin. First we see that the smallness conditions of $u$ and $f$ in lemma \eqref{iterative} is not restrictive. For that, set
 $$
 v(x,t) = \lambda u (\lambda^{a}x, \lambda^{a-1 - \log_{\lambda}(g(\lambda^{(-a-1)}))}t )
 $$
with $\lambda$, $a$ to be fixed, which is a solution of

\begin{eqnarray*}
v_{t} - \textrm{div} \left(g_{d}(|\nabla v(x,t)|) \frac{\nabla v(x,t)}{|\nabla v(x,t)|}\right)
 & = & \tilde{f}(x,t),
		\end{eqnarray*}
where $d= \lambda^{-a-1}$ and
 
$$
\tilde{f}(x,t) = \lambda^{(a - \log_{\lambda}(g(\lambda^{(a-1)}))} f(\lambda^{a}x, \lambda^{a-1 - \log_{\lambda}(g(\lambda^{(-a-1)}))}t).
$$
Choosing $a > 0$ such that
$$
  (a+(a+1))g_{0})r - \frac{an}{1+f_0}r -[a-1 - \log_{\lambda}(g(\lambda^{(-a-1)}))] > 0,
$$
holds, and taking $0 <\lambda \ll 1$, we enter into the smallness conditions
$$
    \| \tilde{f}(x,t) \|_{L^{G,r}(Q_{1})} \le \varepsilon \ \ \text{and} \ \ 
||v||_{\infty, Q_1} \le 1.
$$
Next, take $r \in (0, \rho)$ and choose $k \in \mathbb{N}$  such that $\rho^{k+1} < r \le \rho^{k}$ we estimate, 
\begin{eqnarray*}
\label{sup}
\underset{Q_{r}}{\sup} \vert u(x,t) - u(0,0) \vert &\leq& \underset{Q_{\rho^{k}}}{\sup} \vert u(x,t) - u(0,0) \vert \\
&\le& (\rho^{k})^\alpha \\
&\le& {\left( \frac{r}{\rho} \right)}^{\alpha} = C r^{\alpha}.
\end{eqnarray*}

\end{document}